\documentclass[a4paper,10pt]{amsart}

\usepackage[cp850]{inputenc}

\usepackage{latexsym}
\usepackage{amsfonts}
\usepackage{amsmath}
\usepackage{amssymb,amsxtra,amscd}
\usepackage{mathrsfs}
\usepackage[unicode=true, bookmarks=true, bookmarksnumbered=false, bookmarksopen=false, breaklinks=false, pdfborder={0 0 0}, pdfborderstyle={}, colorlinks=true, linkcolor=blue, citecolor= blue]{hyperref}

\newcommand{\R}{\mathbb{R}}
\renewcommand{\C}{\mathbb{C}}
\newcommand{\N}{\mathbb{N}}

\newcommand{\supp}{\mbox{supp}\,}
\newcommand{\singsupp}{\mbox{sing\,supp}\,}
\newcommand{\dist}{\mbox{dist}}
\newcommand{\proj}{\mbox{proj}}
\newcommand{\ind}{\mbox{ind}}

\newtheorem{theorem}{Theorem}
\newtheorem{lemma}[theorem]{Lemma}
\newtheorem{corollary}[theorem]{Corollary}
\newtheorem{proposition}[theorem]{Proposition}

\newtheorem{remark}[theorem]{Remark}
\newtheorem{definition}[theorem]{Definition}

\title[Surjectivity of differential operators]{Surjectivity of differential operators and linear topological invariants for spaces of zero solutions}
\author{T. Kalmes}
\thanks{{\it E-mail address}: thomas.kalmes@mathematik.tu-chemnitz.de}

\begin{document}

\begin{abstract}
We provide a sufficient condition for a linear differential operator with constant coefficients $P(D)$ to be surjective on $C^\infty(X)$ and $\mathscr{D}'(X)$, respectively, where $X\subseteq\R^d$ is open. Moreover, for certain differential operators this sufficient condition is also necessary and thus a characterization of surjectivity for such differential operators on $C^\infty(X)$, resp.\ on $\mathscr{D}'(X)$, is derived. Additionally, we obtain for certain surjective differential operators $P(D)$ on $C^\infty(X)$, resp.\ $\mathscr{D}'(X)$, that the spaces of zero solutions $C_P^\infty(X)=\{u\in C^\infty(X);\, P(D)u=0\}$, resp.\ $\mathscr{D}_P'(X)=\{u\in\mathscr{D}'(X);\,P(D)u=0\}$ possess the linear topological invariant $(\Omega)$ introduced by Vogt and Wagner in \cite{Vogt4}, resp.\ its generalization $(P\Omega)$ introduced by Bonet and Doma\'nski in \cite{Bonet}.\\


\end{abstract}

\maketitle

\section{Introduction}

A classical result by Malgrange \cite[Chapitre 1, Th\'eor\`eme 4]{Malgrange} from 1955 states that for a polynomial $P\in\C[X_1,\ldots,X_n]$ and for an open set $X\subseteq\R^n$ the constant coefficient differential operator $P(D):C^\infty(X)\rightarrow C^\infty(X)$ is surjective if and only if $X$ is $P$-convex for supports, that is, if and only if for every compact subset $K$ of $X$ there is another compact subset $L$ of $X$ such that for each compactly supported distribution $u\in\mathscr{E}'(X)$ with $\supp P(-D)u\subseteq K$ it holds $\supp u\subseteq L$.

Although this characterization is more than 60 years old, there are very few differential operators $P(D)$ for which there is a satisfactory geometric evaluation of this condition for open $X\subseteq\R^n$. Convex open sets are $P$-convex for supports, whenever $P\neq 0$, every open set is $P$-convex for supports whenever $P$ is elliptic, and there is a complete geometric characterization of $P$-convexity for supports in the two dimensional case due to H\"ormander \cite[Theorem 10.8.3]{Hoermander}. In arbitrary dimensions however, the problem of characterising open subsets of $\R^n$ which are $P$-convex for supports is far from being solved. For second order operators with principal part equal to the wave operator there is a characterization \cite[Theorem 10.8.6]{Hoermander} which is essentially due to Persson \cite{Persson} who extended this result in \cite{Persson2} to arbitrary operators of real principle type but only in $\R^3$, while Tintarev \cite{Tintarev2}, \cite{Tintarev1} evaluated $P$-convexity for supports for operators of real principal type for bounded open sets $X$ with analytic boundary, and for open sets $X$ whose boundary does not contain any straight line, respectively. Moreover, Nakane \cite{Nakane} gave a geometric characterization when $X$ is $P$-convex for supports for the special case of polynomials $P$ acting along a subspace and being elliptic there. A more convenient geometric characterization for this class of operators can be found in \cite[Theorem 10.8.5]{Hoermander}. Apart from these very special classes of polynomials, resp.\ operators, not much is known.

In 1962, H\"ormander showed \cite{Hoermander4} that $P(D)$ is surjective on $\mathscr{D}'(X)$ if and only if $X$ is $P$-convex for supports as well as $P$-convex for singular supports, where the latter means that for every compact subset $K$ of $X$ there is another compact subset $L$ of $X$ such that for each $u\in\mathscr{E}'(X)$ with $\singsupp P(-D)u\subseteq K$ it follows that $\singsupp u\subseteq L$. Thus, surjectivity of $P(D)$ on $\mathscr{D}'(X)$ implies surjectivity of $P(D)$ on $C^\infty(X)$, but in general the converse implication is far from being true, except in case of $X\subseteq\R^2$ as was recently shown in \cite{Kalmes}. There, the author gives a positive solution of a conjecture by Tr\`eves \cite[p.\ 389, Problem 2]{Treves3} showing that for open $X\subseteq\R^2$ $P$-convexity for supports implies $P$-convexity for singular supports. Though better understood than $P$-convexity for supports, giving geometric characterizations for $P$-convexity for singular supports for concrete differential operators $P(D)$ is not a triviality, see e.g.\ \cite[Chapter 10]{Hoermander}, \cite{Hoermander2}, \cite{Frerick}, \cite{Kalmes}, and \cite{Kalmes1}.

Despite the lack of a satisfactory characterization of $P$-convexity for (singular) supports in many cases of operators and thus a characterization of general solvability of linear partial differential equations $P(D)u=f$ for arbitary smooth functions/distributions $f$ on open subsets $X\subseteq\R^n$ there is a number of results, both classical and very recent, by several authors considering the question whether surjectivity of $P(D)$ on $C^\infty(X)$, resp.\ $\mathscr{D}'(X)$ passes on to surjectivity of $P(D)$ on the space of vector-valued smooth functions $C^\infty(X,E)$, resp.\ vector-valued distributions $\mathscr{D}'(X,E)$, where $E$ is a locally convex space (see, e.g.\ \cite{Grothendieck}, \cite{Vogt}, \cite{Bonet}, \cite{Domanski2}, \cite{Bonet2}). Clearly, this question has a positive answer whenever $P(D)$ admits a continuous linear right inverse which was characterized by Meise, Taylor, and Vogt in \cite{MeiseTaylorVogt} (see also \cite{MeiseTaylorVogt2}) solving a problem posed by Schwartz.

In case $E$ is a space of functions or distributions itself, surjectivity of $P(D)$ on $C^\infty(X,E)$, resp.\ $\mathscr{D}'(X,E)$ is equivalent to an affirmative answer to the problem of parameter dependence for solutions of the differential equation $P(D)u_\lambda=f_\lambda$, i.e.\ to the problem whether for a given family $f_\lambda$ of smooth functions, resp.\ distributions on $X$ depending on a parameter $\lambda$ such that the function/distribution $\lambda\mapsto f_\lambda(x)$ belongs to $E$ it is always possible to find solutions $u_\lambda$ such that $\lambda\mapsto u_\lambda(x)$ also belongs to $E$.

For Fr\'echet spaces $E$ it follows from a result of Grothendieck \cite{Grothendieck} that for a surjective differential operator $P(D)$ on $C^\infty(X)$ the operator is also surjective on $C^\infty(X,E)$. The same implication is no longer true in general when $E$ is the strong dual of a Fr\'echet space $F$, as has been shown by Vogt \cite{Vogt}. However, Vogt also showed that in case of $F=s$, with $s$ denoting the space of rapidly decreasing sequences, surjectivity of  $P(D)$ on $C^\infty(X)$ passes on to surjectivity on $C^\infty(X,s')$ if and only if the kernel of $P(D)$ in $C^\infty(X)$, which we denote by $C^\infty_P(X)$ and which is equipped with the Fr\'echet space structure inherited by $C^\infty(X)$, has the linear topological invariant $(\Omega)$ (see \cite[Proposition 2.2, Theorem 2.4]{Vogt}). Recall that a Fr\'echet space $F$ with a fundamental system of seminorms $\|\cdot\|_1\leq \|\cdot\|_2\leq\ldots$ is said to have $(\Omega)$ if the following holds
\begin{equation}\label{omega}
\forall\,k\,\exists\,l\,\forall\,m\,\exists\,\theta\in (0,1), C>0\,\forall\,u\in F':\,\|u\|_l^*\leq C\|u\|_k^{*\,(1-\theta)}\|u\|_m^{*\,\theta},
\end{equation}
where $\|\cdot\|_k^*$ etc.\ denote the so-called dual norm on $F'$ to $\|\cdot\|_k$, i.e.\ $\|u\|_k^*=\sup\{|u(x)|;\,x\in E,\|x\|_k\leq 1\}$. Obviously, property $(\Omega)$ is a linear topological invariant, in particular it is independent of the special choice of fundamental system of seminorms on $E$.

For hypoelliptic polynomials $P$ property $(\Omega)$ can be viewed as an abstract qualitative version of Hadamard's three circles theorem. In fact, having in mind Grothen\-dieck-K\"othe duality (see e.g.\ \cite{Grothendieck2}, \cite[p. 175]{Vogt}, \cite{Wiechert}) by which for surjective $P(D)$ on $C^\infty(X)$ the strong dual space of $C^\infty_P(X)$ is given in a natural way by the space of germs of smooth functions $u$ on the complement of $X$ which satisfy $P(-D)u=0$ and which behave well at infinity and taking into account that in the case of hypoelliptic $P(D)$ the dual norms appearing in (\ref{omega}) can be chosen as (polynomially) weighted supremum norms over the complement of compact subsets of $X$, condition (\ref{omega}) becomes in fact a very concrete generalisation of Hadamard's result for the Cauchy-Riemann operator to arbitrary hypoelliptic differential operators.

Apart from being characteristic for when surjectivity of $P(D)$ on $C^\infty(X)$ passes on to surjectivity of $P(D)$ on $C^\infty(X,s')$ it has also been shown by Vogt in \cite{Vogt} that property $(\Omega)$ of $C^\infty_P(X)$ is sufficient for $P(D)$ to be surjective on $C^\infty(X,F')$ provided that the Fr\'echet space $F$ has property $(DN)$, i.e.\ the following condition holds 
\[\exists\,k\,\forall\,l\,\exists\,m, C>0\,\forall\,x\in F:\,\|x\|_l^2\leq C\|x\|_k\|x\|_m,\]
where $\|\cdot\|_1\leq \|\cdot\|_2\leq\ldots$ is again a fundamental system of seminorms. Examples of Fr\'echet spaces with $(DN)$ are $s$, $\mathscr{S}(\R^n)$ the space rapidly decreasing functions, $C^\infty_{2\pi}(\R^m)$ the space of smooth functions which are $2\pi$-periodic in each variable, $H(\C)$ the space of entire functions, or more generally Fr\'echet power series spaces of infinite type $\Lambda_\infty(\alpha)$ (see e.g.\ \cite{MeiseVogt}).

%Apart from being sufficient condition for surjectivity of $P(D)$ on $C^\infty(X)$ to pass on to surjectivity of $P(D)$ on $C^\infty(X,F')$, where $F$ is a Fr\'echet space with property $(DN)$ in terms of $C_P^\infty(X)$, property $(\Omega)$ - together with property $(DN)$ - plays a prominent role in the stucture theory of nuclear Fr\'echet spaces: a Fr\'echet space $F$ is (topologically) isomorphic to a quotient of $s$ if and only if $F$ is nuclear and has property $(\Omega)$ (see \cite{Vogt3}, \cite{Vogt4}).

In \cite{Vogt} it has been proved that for elliptic operators $P$ the kernel $C^\infty_P(X)$ always has $(\Omega)$ and that the same holds in the more general case of hypo\-elliptic $P$ but under the restriction that the open set $X\subseteq\R^n$ is convex. Therefore it was conjectured for a long time that for a hypoelliptic operator $P(D)$ which is surjective on $C^\infty(X)$ its kernel $C^\infty_P(X)$ always has $(\Omega)$ independent of $X$ being convex or not. However, only recently, this conjecture has been settled in the negative. In \cite{Kalmes2} for any $n\geq 3$ an example of a hypoelliptic operator $P(D)$ and a $P$-convex open set $X\subseteq\R^n$ has been given such that $C^\infty_P(X)$ does not have $(\Omega)$. (That $n\geq 3$ is essential here follows from results obtained in \cite{Kalmes1}.) Stronger variants of the linear topological invariant $(\Omega)$ for $C_P^\infty(X)$ have been considered in \cite{Vogt5}.

It will be shown in section \ref{topological invariants} below that the kernels $C^\infty_P(X)$ of semi-elliptic differential operators $P(D)$ for which the zero set of the principal part is a one dimensional subspace of $\R^n$ always has $(\Omega)$ whenever $X$ is $P$-convex. This class of differential operators contains in particular non-degenerate parabolic operators which is obviously of particular interest in view of applications to concrete problems.

The problem of surjectivity of $P(D)$ on the space of vector-valued distributions $\mathscr{D}'(X,E)$ instead of vector-valued smooth functions was only addressed recently by Bonet and Doma\'nski in \cite{Bonet}. They proved that for a surjective differential operator $P(D)$ on $\mathscr{D}'(X)$ the operator $P(D)$ is surjective on $\mathscr{D}'(X,F')$ for a nuclear Fr\'echet space $F$ with property $(DN)$ if the kernel of $P(D)$ in $\mathscr{D}'(X)$, which we denote by $\mathscr{D}_P'(X)$ and which we equip with the locally convex topology inherited by the strong topology on $\mathscr{D}'(X)$, has property $(P\Omega)$, a generalisation of $(\Omega)$ to the category of PLS-spaces. Recall that a locally convex space $E$ is a PLS-space if $E$ is the projective limit of a sequence of strong duals of Fr\'echet-Schwartz spaces, $E=\proj_{N\in\N}X_N, X_N=\ind_{n\in\N} X_{N,n}$, where $X_{N,n}$ are Banach spaces and the linking maps in the inductive limit $\ind_{n\in\N}X_{N,n}$ are compact. A PLS-space $E$ has property $(P\Omega)$ if
\begin{eqnarray*}
	&\forall\, K\,\exists\,L\geq K\,\forall\, M\geq L\,\exists\,k\,\forall\,l\,\exists\,m, C>0,\theta\in (0,1)\,\forall\,u\in X_K':\\
	&\|u'\circ i_K^L\|_{L,l}^*\leq C\|u'\circ i_K^M\|_{M,m}^{*\,(1-\theta)}\max\{\|u'\|_{K,k}^{*\,\theta},\|u'\circ i_K^M\|_{M,m}^{*\,\theta}\},
\end{eqnarray*}
where $\|\cdot\|_{L,l}^*$ denotes the dual norm to the norm on the Banach space $X_{L,l}$ etc.\ and $i_K^L$ denotes the inclusion from $X_K$ into $X_L$. It is well-known that Fr\'echet-Schwartz spaces are PLS-spaces and examples of PLS-spaces are thus spaces of holomorphic functions $H(V)$, smooth functions $C^\infty(X)$, or more general spaces of distributions $\mathscr{D}'(X)$, real analytic functions $\mathscr{A}(X)$, ultradistributions in the sense of Beurling $\mathscr{D}'_{(\omega)}(X)$, or ultradifferentiable functions in the sense of Roumieu $\mathscr{E}_{\{\omega\}}(X)$. See e.g.\ \cite{Domanski} for more on PLS-spaces, their properties, and their importance in analysis. It is shown in \cite{Bonet} that a Fr\'echet-Schwartz space has $(P\Omega)$ if and only if it has $(\Omega)$. Recall that for hypoelliptic differential operators $P(D)$ the kernels $C_P^\infty(X)$ and $\mathscr{D}'_P(X)$ coincide as locally convex spaces by \cite[Theorem 52.1]{Treves2} so that $\mathscr{D}'_P(X)$ is a Fr\'echet-Schwartz space and thus it has $(P\Omega)$ if and only if it has $(\Omega)$.\\

The present article is organized as follows. In section 2 we give sufficient conditions for $P$-convexity for (singular) supports for $X$ in terms of the minimum principle for the boundary distance of $X$ being valid in certain affine subspaces related to $P$. This enables to give geometric characterizations for $X$ for the surjectivity of $P(D)$ on $C^\infty(X)$ and $\mathscr{D}'(X)$, respectively, for certain types of differential operators.

In section 3 the sufficient condition for $P$-convexity for singular supports is used to give for certain differential operators a positive solution to the problem whether for a surjective operator $P(D)$ on $\mathscr{D}'(X)$ the augmented operator $P^+(D)$ is surjective on $\mathscr{D}'(X\times\R)$, where $P^+(x_1,\ldots,x_{n+1})=P(x_1,\ldots,x_n)$. This problem was posed by Bonet and Doma\'nski in \cite{Bonet}. Although in general this problem has a negative solution as shown in \cite{Kalmes2}, in section 3 we show that this problem has always a positive solution for certain semi-elliptic differential operators including the heat operator, and for operators acting along a subspace of $\R^n$ and being elliptic there.

In section 4 we apply the results from section 3 to the differential operators considered there in order to show that the spaces of zero solutions $C^\infty_P(X)$, resp.\ $\mathscr{D}'_P(X)$ have property $(\Omega)$, resp.\ $(P\Omega)$ whenever $P(D)$ is surjective on $C^\infty(X)$, resp.\ $\mathscr{D}'(X)$ so that for certain locally convex spaces $E$ $P(D)$ is also surjective on $C^\infty(X,E)$, resp.\ on $\mathscr{D}'(X,E)$.

Throughout the paper we use standard notation from the theory of partial differential operators, see e.g.\ \cite{Hoermander}, and functional analysis, see e.g.\ \cite{MeiseVogt}.

\section{Conditions for $P$-convexity}

It is well-known that a necessary condition for an open subset $X\subseteq\R^n$ to be $P$-convex for supports, respectively $P$-convex for singular supports, is that the boundary distance of $X$ satisfies the minimum principle in every characteristic hyperplane, respectively in certain affine subspaces related to $P$, resp.\ (cf.\ \cite[Theorem 10.8.1, Corollary 11.3.2, resp.]{Hoermander}). Recall that a real-valued function $f$ on $X\subseteq\R^n$ is said to satisfy the minimum principle in a closed subset $F$ of $\R^n$ if for every compact set $K\subseteq F\cap X$ we have
\[\min_{x\in K}f(x)=\min_{\partial_F K}f(x),\]
where $\partial_F K$ denotes the boundary of $K$ in $F$. We set the boundary distance $d_X$ of $X$ to be the mapping
\[d_X:X\rightarrow\R,x\mapsto\dist(x,X^c),\]
where the distance is taken with respect to the Euclidean norm $|x|$ in $\R^n$. The next theorem gives sufficient conditions for $P$-convexity for (singular) supports in terms of $d_X$ satisfying the minimum principle in certain affine subspaces related to $P$, and it is the main result of this section. In its formulation we use the following functional defined on the subspaces of $\R^n$. It was introduced by H\"ormander in the context of continuation of differentiability for solutions of partial differential equations, cf.\ \cite[Section 11.3]{Hoermander}.

For a subspace $V\subseteq\R^n$ we set
\[\sigma_P(V)=\inf_{t\geq 1}\liminf_{\xi\rightarrow\infty}\frac{\tilde{P}_V(\xi,t)}{\tilde{P}(\xi,t)},\]
where for $t\geq 1$ and $\xi\in\R^n$
\[\tilde{P}_V(\xi,t):=\sup_{x\in V,|x|\leq t}|P(x+\xi)|,\,\tilde{P}(\xi,t):=\tilde{P}_{\R^n}(\xi,t).\]
Moreover, for $x\in\R^n$ we use the abbreviation $\sigma_P(x):=\sigma_P(\mbox{span}\{x\})$.

\begin{theorem}\label{minimal sufficient}
Let $P$ be a polynomial with principal part $P_m$ and let $X\subseteq\R^n$ be open. Moreover, let $W\neq\{0\}$ be a subspace of $\R^n$ such that $d_X$ satisfies the minimum principle in $x+W$ for every $x\in\R^n$.
\begin{itemize}
	\item[i)] If $\{x\in\R^n;P_m(x)=0\}\subseteq W^\perp$ then $X$ is $P$-convex for supports.
	\item[ii)] If $\{x\in\R^n;\sigma_P(x)=0\}\subseteq W^\perp$ then $X$ is $P$-convex for singular supports.
\end{itemize}
\end{theorem}

With $W=\R^n$ part i) of the above theorem gives a new proof of the well-known result that for elliptic $P$ every open set $X\subseteq\R^n$ is $P$-convex for supports. $P$-convexity for supports for polynomials for which the zero set of its principal part is contained in a non-trivial subspace of $\R^n$ was also considered by Zachmanoglou \cite{Zachmanoglou}. In order to prove Theorem \ref{minimal sufficient} we first give a useful characterization of when $d_X$ satisfies the minimum principal in every affine subspace parallel to a given subspace $W$. Therefore, we introduce the following notion.

\begin{definition}
\begin{rm}
Let $W$ be a subspace of $\R^n$ and $X\subseteq\R^n$. For an interval $I\subseteq\R$ a curve $\gamma:I\rightarrow X$ is called $W$-\textit{regular} if $\gamma$ is continuous and piecewise continuously differentiable with $\gamma'(t)\in W$ for all $t\in I$ where $\gamma$ is differentiable.
\end{rm}
\end{definition}

For $X\subseteq\R^n$ we denote by $\partial_\infty X$ the boundary of $X$ in the one-point compactification of $\R^n$, thus $\infty\in\partial_\infty X$ whenever $X$ is an unbounded subset of $\R^n$. For a curve $\gamma:I\rightarrow X$ we set $|\gamma|:=\gamma(I)$

\begin{proposition}\label{geometry}
Let $W$ be a subspace of $\R^n$, $X\subseteq\R^n$ be open, $K\subseteq X$ be compact and non-empty, and $x\in X\backslash K$ such that $|\gamma|\cap K\neq\emptyset$ for every $W$-regular curve $\gamma:[0,\infty)\rightarrow X$ with $\gamma(0)=x$ and $\liminf_{t\rightarrow\infty}\dist(\gamma(t),\partial_\infty X)=0$.

Then, the connected component $Z$ in $(x+W)\cap (X\backslash K)$ which contains $x$ is bounded and $\partial_{x+W}Z$, the boundary of $Z$ in $x+W$, is contained in $K\cap(x+W)$.
\end{proposition}

\begin{proof}
Assume that $Z$ is unbounded. For notational convenience let us define $V:=(x+W)\cap (X\backslash K)$. Because $V$ is locally pathwise connected and $Z$ is open in $V$, thus locally pathwise connected, too, the connectedness of $Z$ implies that $Z$ is pathwise connected.  Hence, there is a continuous, piecewise continuously differentiable curve $\gamma:[0,\infty)\rightarrow Z$ with $\gamma(0)=x$ and $\lim_{t\rightarrow\infty}|\gamma(t)|=\infty$. Because $Z\subseteq X\backslash K$ this contradicts the hypothesis.

Thus, $Z$ is a bounded subset of $V$. Because the inclusion $V\hookrightarrow x+W$ is a homeomorphism onto its image and because $Z$ is open and closed in $V$ so that $\partial_V Z=\emptyset$, we have
\begin{eqnarray*}
\partial_{x+W}Z&=&\Big((\partial_{x+W}Z)\cap V\Big)\cup\Big((\partial_{x+W}Z)\cap(K\cap(x+W))\Big)\\
&&\cup\Big((\partial_{x+W}Z)\cap((\R^n\backslash X)\cap(x+W))\Big)\\
&=&\partial_V Z\cup\Big((\partial_{x+W}Z)\cap(K\cap(x+W))\Big)\\
&&\cup\Big((\partial_{x+W}Z)\cap((\R^n\backslash X)\cap(x+W))\Big)\\
&=&\Big((\partial_{x+W}Z)\cap(K\cap(x+W))\Big)\\
&&\cup\Big((\partial_{x+W}Z)\cap((\R^n\backslash X)\cap(x+W))\Big).
\end{eqnarray*}
So, if we can show that $(\partial_{x+W}Z)\cap((\R^n\backslash X)\cap(x+W))=\emptyset$ the proposition is proved.

Assume that there is
\[y\in(\partial_{x+W}Z)\cap((\R^n\backslash X)\cap(x+W)).\] Because $V\hookrightarrow x+W$ is a homeomorphism onto its image it follows that $Z$ is a pathwise connected subset of $x+W$. Denote by $C$ the closure of $Z$ in $x+W$. Since $Z$ is open in $x+W$ it follows that there is a continuous, piecewise continuously differentiable curve $\gamma:[0,\infty)\rightarrow Z$ with $\gamma(0)=x$ and $\lim_{t\rightarrow\infty}\gamma(t)=y$.

Indeed, there is $y_1$ in $B(y,1)\cap Z$ and a continuous, piecewise continuously differentiable $\alpha_1:[0,1]\rightarrow Z$ with $\alpha_1(0)=x$, $\alpha_1(1)=y_1$. Next, as we can find  $y_2\in B(y,1-|y_1|)\cap Z$ there is a continuous, piecewise continuously differentiable $\alpha_2:[1,2]\rightarrow Z$ with $\alpha_2(1)=y_1$ and $\alpha_2(2)=y_j$. Preceding in this way we obtain a sequence $(y_j)_{j\in\N}$ in $Z$ with $y=\lim_{j\rightarrow\infty}y_j$ and a sequence of continuous, piecewise continuously differentiable curves $\alpha_j:[j-1,j]\rightarrow Z$ with $\alpha_j(j-1)=y_{j-1}$ and $\alpha_j(j)=y_j$. Joining these curves gives the desired $\gamma$.

Clearly, $\gamma$ is a $W$-regular curve with $\gamma(0)=x$ and $\liminf_{t\rightarrow\infty}\dist(\gamma(t),X^c)=0$. From the hypothesis it follows $|\gamma|\cap K\neq\emptyset$ contradicting $Z\subseteq X\backslash K$. 
\end{proof}

We now give a characterization of when $d_X$ satisfies the minimum principle in every affine subspace parallel to a given subspace $W$ in terms of $W$-regular curves.

\begin{lemma}\label{equivalent}
Let $X\subseteq\R^n$ be open and let $\{0\}\neq W\subseteq\R^n$ be a subspace. Then the following are equivalent.
\begin{itemize}
\item[i)] The boundary distance
\[d_X:X\rightarrow\R,x\mapsto\dist(x,X^c)\]
satisfies the minimum principle in $x+W$ for each $x\in X$. \item[ii)] For each compact subset $K\subseteq X$ and every
\[x\in \{y\in X;\,d_X(y)<\dist(K,X^c)\}\]
there is a $W$-regular curve $\gamma:[0,\infty)\rightarrow X$ with $\gamma(0)=x$ and $|\gamma|\cap K=\emptyset$ such that $\liminf_{t\rightarrow\infty}\dist(\gamma(t),\partial_\infty X)=0$.
%\begin{eqnarray*}
%&&\exists\,W\mbox{-regular curve }\gamma:[0,\infty)\rightarrow X, \gamma(0)=x,\\
%&&\liminf_{t\rightarrow\infty}\dist(\gamma(t),\partial_\infty X)=0:\,|\gamma|\cap K=\emptyset.
%\end{eqnarray*}
\end{itemize}
\end{lemma}

\begin{proof}
We first show that ii) implies i). So let $K\subseteq (x+W)\cap X$ be compact for some $x\in X$. Assume that there is $x_0\in K$ with $d_X(x_0)<\dist(\partial_{x+W}K,X^c)$. Applying ii) to the compact subset $\partial_{x+W}K$ of $X$ there is a $W$-regular curve $\gamma:[0,\infty)\rightarrow X$ with $\gamma(0)=x_0$, $\liminf_{t\rightarrow\infty}\dist(\gamma(t),\partial_\infty X)=0$ and $|\gamma|\cap\partial_{x+W}K=\emptyset$. Because $\gamma$ is $W$-regular and because $x_0\in K\subseteq x+W$ it follows that $|\gamma|\subseteq (x+W)\cap X$. Since $\gamma(0)=x_0$ is contained in $\mbox{int}_{x+W}(K)$, the interior of $K$ with respect to $x+W$, and since $\liminf_{t\rightarrow\infty}\dist(\gamma(t),\partial_\infty X)=0$ it follows from the compactness of $K\subseteq X$ that $|\gamma|\cap\partial_{x+W}K\neq\emptyset$ which gives a contradiction.

In order to show that i) implies ii), assume that for some compact $K\subseteq X$ there is $x_0\in \{y\in X;\,d_X(y)<\dist(K,X^c)\}$ such that every $W$-regular curve $\gamma:[0,\infty)\rightarrow X$ with $\gamma(0)=x_0$ and $\liminf_{t\rightarrow\infty}\dist(\gamma(t),\partial_\infty X)=0$ intersects $K$.

Let $Z$ be the connected component of $(x_0+W)\cap(X\backslash K)$ containing $x_0$. From the assumption on $x_0$ and proposition \ref{geometry} it follows that $Z$ is a bounded subset of $x_0+W$ and
\begin{equation}\label{boundary}
	\partial_{x_0+W}Z\subseteq K.
\end{equation}
Let us denote the closure of $Z$ in $x_0+W$ by $C$. Then $C$ is a compact subset of $(x_0+W)\cap X$ by proposition \ref{geometry}. By i) we obtain
\[\min_{x\in C}d_X(x)=\min_{x\in\partial_{x_0+W}C}d_X(x).\]
Because of $x_0\in C$ this yields
\[d_X(x_0)\geq\min_{x\in\partial_{x_0+W}C}d_X(x) =\dist(\partial_{x_0+W}Z,X^c)\geq\dist(K,X^c) \]
where we used (\ref{boundary}) in the last inequality. But this contradicts $d_X(x_0)<\dist(K,X^c)$.
\end{proof}

\textit{Proof of Theorem \ref{minimal sufficient}}.
In the sequel we denote the convex hull of $x,y\in\R^n$ by $[x,y]$. In order to prove i) we fix $u\in\mathscr{E}'(X)$ and set $K:=\supp \check{P}(D)u$, where as usual $\check{P}(\xi)=P(-\xi)$. Moreover, we fix $x\in\{y\in X;\,\dist(y,X^c)<\dist(K,X^c)\}$. Because $d_X$ satisfies the minimum principle in $y+W$ for every $y\in X$ it follows from lemma \ref{equivalent} that there is a $W$-regular curve $\gamma:[0,\infty)\rightarrow X$ such that $\gamma(0)=x$, $|\gamma|\cap K=\emptyset$, and $\liminf_{t\rightarrow\infty}\dist(\gamma(t),X^c)=0$.

Because $\supp u$ is a compact subset of $X$ it follows from the properties of $\gamma$ that there is $T>0$ with $\gamma(T)\notin\supp u$. Moreover, we can find $\varepsilon>0$ such that the open ball $B(\gamma(T),\varepsilon)$ of radius $\varepsilon$ about $\gamma(T)$ does not intersect $\supp u$, $\gamma([0,T])+B(0,\varepsilon)\subseteq X$ and $K\cap(\gamma([0,T])+B(0,\varepsilon))=\emptyset$, where $\gamma([0,T])+B(0,\varepsilon)=\{y+z;y\in\gamma([0,T]),z\in B(0,\varepsilon)\}$.

Next, we choose $0=t_0<t_1<\ldots<t_k=T$ such that for each $j=1,\ldots,k$ the restriction of $\gamma$ to $[t_{j-1},t_j]$ is continuously differentiable and
\[|\int_{t_{j-1}}^{t_j}\gamma'(t)dt|<\frac{\varepsilon}{2}.\]
We define
\[f:[0,k]\rightarrow\R^n,s\mapsto\gamma(t_{\lfloor s\rfloor})+(s-\lfloor s\rfloor)\int_{t_{\lfloor s\rfloor}}^{t_{\lfloor s\rfloor}+1}\gamma'(t)dt,\]
where $\lfloor s\rfloor$ denotes the integer part of $s$. Then $f$ is a polygonal curve in $x+W$ by the $W$-regularity of $\gamma$ with $f([j-1,j])=[\gamma(t_{j-1}),\gamma(t_j)], j=1,\ldots,k$. Moreover, due to the choice of $\varepsilon$, we have $|f|+B(0,\frac{\varepsilon}{2})\subseteq X\backslash K$.

For $N\in\{y\in\R^n;\,P_m(y)=0\}\backslash\{0\}$ and $\alpha\in\R$ let
\[H_{N,\alpha}=\{x\in\R^n;\,\langle x,N\rangle=\alpha\}\]
be the corresponding characteristic hyperplane. Since $[\gamma(t_{k-1}),\gamma(T)]\subseteq x+W$ and $N\in W^\perp$ it follows that $H_{N,\alpha}$ intersects $B(\gamma(T),\varepsilon)$ whenever $H_{N,\alpha}$ intersects $[\gamma(t_{k-1}),\gamma(T)]+B(0,\varepsilon)$. By the choice of $\varepsilon$ we have $u_{|B(\gamma(T),\varepsilon)}=0$ so that by \cite[Theorem 8.6.8, vol.\ I]{Hoermander} $u$ vanishes in $[\gamma(t_{k-1}),\gamma(T)]+B(0,\varepsilon)$.

Iteration of the above argument yields that $u$ vanishes in $|f|+B(0,\varepsilon)$, in particular $x=\gamma(0)=f(0)$ does not belong to $\supp u$. As $x$ was an arbitrary point with $\dist(x,X^c)<\dist(K,X^c)$ it follows from the definition of $K$ that
\[\dist(\supp u,X^c)\geq \dist(K,X^c)=\dist(\supp \check{P}(D),X^c).\]
Since trivially $\dist(\supp u,X^c)\leq\dist(\supp\check{P}(D)u,X^c)$ it follows from \cite[Theorem 10.6.3, vol.\ II]{Hoermander} that $X$ is $P$-convex for supports.

In order to prove ii), we replace in the above arguments $\supp$ by $\singsupp$, $N\in\{y\in\R^n;\,P_m(y)=0\}$ by $N\in\{y\in\R^n;\,(\sigma_{\check{P}}(y)=)\sigma_P(y)=0\}$, and the reference to \cite[Theorem 10.6.3, vol.\ II]{Hoermander} by \cite[Corollary 2]{Frerick}.\hfill$\square$\\

As a corollary to Theorem \ref{minimal sufficient} we obtain the next result characterizing $P$-convexity for supports for particular polynomials $P$. This characterization is in particular applicable to the the Schr\"odinger operator $P(D)=\Delta_x+i\frac{\partial}{\partial t}$ and the heat operator $P(D)=\Delta_x-\frac{\partial}{\partial t}$, or more generally, to non-degenerate parabolic operators.

\begin{corollary}\label{characterization for supports}
Let $X\subseteq\R^n$ be open and let $P$ be a polynomial with principal part $P_m$ such that $\{x\in\R^n;\,P_m(x)=0\}$ is a one-dimensional subspace of $\R^n$. Then, the following are equivalent.
\begin{itemize}
	\item[i)] $P(D)$ is surjective on $C^\infty(X)$.
	\item[ii)] $d_X$ satisfies the minimum principle in every characteristic hyperplane.
\end{itemize} 
\end{corollary}

\begin{proof}
We set $W:=\{x\in\R^n;\,P_m(x)=0\}^\perp$. That ii) implies i) follows from Theorem \ref{minimal sufficient} while the reverse implication is \cite[Theorem 10.8.1, vol.\ II]{Hoermander}.
\end{proof}

The analogous result for $P$-convexity for singular supports is the following.

\begin{corollary}\label{flat singular}
Let $P\neq 0$ be a polynomial such that $\{x\in\R^n;\,\sigma_P(x)=0\}=W^\perp$ for some subspace $W\subseteq\R^n$ with $\sigma_P(W^\perp)=0$. Then, for $X\subseteq\R^n$ the following are equivalent.
\begin{itemize}
	\item[i)] $X$ is $P$-convex for singular supports.
	\item[ii)] $d_X$ satisfies the minimum principle in $x+W$ for every $x\in\R^n$.
\end{itemize}
\end{corollary}

\begin{proof}
That ii) implies i) follows again from Theorem \ref{minimal sufficient} while the reverse implication is \cite[Corollary 11.3.2, vol.\ II]{Hoermander}.
\end{proof}

\begin{remark}
\begin{rm}
We do not know if the condition "$\sigma_P(W^\perp)=0$" in Corollary \ref{flat singular} is redundant.
\end{rm}
\end{remark}

Recall that a polynomial $P\in\C[X_1,\ldots,X_d]$ is said to act along a subspace $W$ of $\R^d$ if $P(x)=P(\pi_W(x))$ for every $x\in\R^d$, where $\pi_W$ denotes the orthogonal projection of $\R^d$ onto $W$. It is easily seen that $P$ acts along a subspace $W$ of $\R^d$ if and only if $P(x+y)=P(x)$ holds for every $x\in\R^d$ and each $y\in W^\perp$. Moreover, a polynomial $P$ acting along a subspace $W$ is said to be elliptic on $W$ if $P_m(x)\neq 0$ for all $x\in W\backslash\{0\}$.

\begin{lemma}\label{elliptic zeros}
Let $P\in\C[X_1,\ldots,X_n]$ be a non-constant polynomial which acts along a subspace $W\subseteq\R^n$ and is elliptic as a polynomial on $W$. Then $\sigma_P(V)=0$ holds for a subspace $V\subseteq\R^n$ if and only if $V\subseteq W^\perp$.
\end{lemma}

\begin{proof}
Without loss of generality we may assume $W=\R^k\times\{0\}$ with $1\leq k\leq n$ and $0\in\R^{n-k}$. For $x=(x_1,\ldots,x_n)$ we denote $x'=(x_1,\ldots,x_k)$.

Clearly, $V\subseteq W^\perp$ implies $\sigma_P(V)=0$. On the other hand, if $V$ is not a subspace of $W^\perp$ there is $x\in V$ with $|x'|=1$. From the ellipticity of $P$ on $\R^k$ it follows that for some $C>0$ we have
\[\forall\,\xi\in\R^n:\,C |\xi'|^m\leq |P(\xi')|,\]
where $m\in\N_0$ is the degree of $P$. Thus, for $\xi\in\R^n$ and $t>0$ we have
\begin{eqnarray*}
	\tilde{P}_{\mbox{span}\{x\}}(\xi,t)^2&=&\sup_{|\theta|\leq t}|P(\xi'+\theta x')|^2\\
	&\geq& C\sup_{|\theta|\leq t}|\xi'+\theta x'|^{2m}\\
	&\geq&\frac{C}{2}\Big(|\xi'+tx'|^{2m}+|\xi'-tx'|^{2m}\Big)\\
	&=&\frac{C}{2}\Big(\big(\sum_{j=1}^k (\xi_j+tx_j)^2\big)^m+\big(\sum_{j=1}^k(\xi_j-tx_j)^2\big)^m\Big)\\
	&\geq&\frac{C}{2}\Big(\sum_{j=1}^k\big((\xi_j+tx_j)^{2m}+(\xi_j-tx_j)^{2m}\big)\Big)\\
	&\geq& C\Big(\sum_{j=1}^k\xi_j^{2m}+t^{2m}\sum_{j=1}^k x_j^{2m}\Big).
\end{eqnarray*}
Moreover, for suitable $C',D>0$ we have for all $\xi\in\R^n$ and $t>0$
\[\tilde{P}(\xi,t)^2\leq C'(t^{2m}+|\xi'|^{2m})\leq D(t^{2m}+\sum_{j=1}^k\xi_j^{2m}).\]
If we set $\alpha:=\sum_{j=1}^k x_j^{2m}$ and take into account that $\alpha>0$ because of $|x'|=1$ it follows that for every $\xi\in\R^n$ and $t>0$ we have
\[\frac{\tilde{P}_{\mbox{span}\{x\}}(\xi,t)^2}{\tilde{P}(\xi,t)^2}\geq\frac{C\Big(\sum_{j=1}^k\xi_j^{2m}+t^{2m}\alpha\Big)}{D(\sum_{j=1}^k\xi_j^{2m}+t^{2m})}\geq\frac{C}{D}\alpha>0,\]
so that
\[\sigma_P(V)\geq\sigma_P(x)\geq\frac{C}{D}\alpha>0\]
which proves the lemma.
\end{proof}

The next theorem complements a result of Nakane \cite{Nakane} mentioned in the introduction.

\begin{theorem}\label{elliptic}
Let $P\in\C[X_1,\ldots,X_n]$ be a polynomial which acts along a subspace $W\subseteq\R^n$ and is elliptic as a polynomial on $W$. Then, for an open subset $X\subseteq\R^n$ the following are equivalent.
\begin{itemize}
	\item[i)] $X$ is $P$-convex for supports.
	\item[ii)] $X$ is $P$-convex for singular supports.
	\item[iii)] $d_X$ satisfies the minimum principle in $x+W$ for every $x\in\R^n$.
\end{itemize}
In particular, $P(D)$ is surjective on $C^\infty(X)$ if and only if $P(D)$ is surjective on $\mathscr{D}'(X)$.
\end{theorem}

\begin{proof}
That i) and iii) are equivalent is \cite[Theorem 10.8.5]{Hoermander} while the equivalence of ii) and iii) follows from Corollary \ref{flat singular} and Lemma \ref{elliptic zeros}.
\end{proof}

\begin{corollary}
Let $P(x)=\alpha\prod_{j=1}^l(\langle N_j,x\rangle+c_j)$ where $N_j\in\C^n\backslash\{0\}, \alpha, c_j\in\C, j=1,\ldots,l,$. Then for an open subset $X\subseteq\R^n$ the following are equivalent.
\begin{itemize}
	\item[i)] $X$ is $P$-convex for supports.
	\item[ii)] $X$ is $P$-convex for singular supports.
	\item[iii)] $d_X$ satisfies the minimum principle in $x+W$ for every $x\in\R^n$, where $W$ is one of the subspaces $\mbox{span}\{\mbox{Re } N_j,\mbox{Im } N_j\}, j=1\ldots,l$.
\end{itemize}
\end{corollary}

\begin{proof}
Since each of the the first order operators $\langle N_j,D\rangle+c_j$ acts along one of the subspaces $\mbox{span}\{\mbox{Re } N_j,\mbox{Im } N_j\}$ and is elliptic there, the corollary follows directly from Theorem \ref{elliptic} and the fact that these operators commute. 
\end{proof}

\section{Surjectivity of certain augmented differential operators}

In this section we show that for certain partial differential operators, including non-degenerate parabolic operators like the heat operator, surjectivity of $P(D)$ on $\mathscr{D}'(X)$ implies the surjectivity of the augmented operator $P^+(D)$ on $\mathscr{D}'(X\times\R)$, where $P^+(x_1,\ldots,x_n,x_{n+1}):=P(x_1,\ldots,x_n)$. It was shown in \cite{Kalmes1}, that in case of $n=2$ the augmented operator of a surjective partial differential operator $P(D)$ is always surjective while in \cite{Kalmes2} for $n\geq 3$ a hypoelliptic differential operator $P(D)$ was constructed for which there is some open $X\subseteq\R^n$ such that $P(D)$ is surjective on $\mathscr{D}'(X)$ while the augmented operator $P^+(D)$ is not surjective on $\mathscr{D}'(X\times\R)$. Thus, although not true in general, for certain differential operators, including the heat operator, the problem of parameter dependence for solutions of partial differential equations \cite[Problem 9.1]{Bonet} has a positive solution, as will be shown as a consequence of the results from this section in section 4.

Apart from the functional $\sigma_P$ we also use
\[\sigma_P^0(V):=\inf_{t\geq 1,\xi\in\R^n}\frac{\tilde{P}_V(\xi,t)}{\tilde{P}(\xi,t)},\]
where $V\subseteq\R^n$ is again a subspace. Moreover, we set $\sigma_p^0(y):=\sigma_P^0(\mbox{span}\{y\})$ for $y\in\R^n$. Furthermore, we define
\[\pi:\R^{n+1}\rightarrow\R^n,(x_1,\ldots,x_{n+1})\mapsto (x_1,\ldots,x_n)\]
and we set $A':=\pi(A)$ for $A\subseteq\R^{n+1}$, respectively $x':=\pi(x)$ for $x\in\R^{n+1}$. The reason for using $\sigma_P^0$ is the following lemma. For its proof see \cite[Lemma 1]{Frerick}.

\begin{lemma}\label{equivalence}
For a polynomial $P\in\C[X_1,\ldots,X_n]$ and a subspace $V\subseteq\R^{n+1}$ the following hold.
\begin{itemize}
	\item[i)] $\sigma_{P^+}(V'\times\R)=\sigma_{P^+}(V'\times\{0\})=\sigma_P^0(V')$.
	\item[ii)] $\sigma_{P^+}(V)=0$ if and only if $\sigma_P^0(V')=0$.
\end{itemize}
\end{lemma}

We continue with an easy geometrical observation. For its proof we need the following trivial proposition.

\begin{proposition}\label{topology}
Let $X$ and $Y$ be Hausdorff topological spaces, $f:X\rightarrow Y$ continuous and open, and let $K\subseteq X$ be compact. Then $\partial f(K)\subseteq f(\partial K)$.
\end{proposition}

\begin{proof}
Let $y\in\partial f(K)$ and $x\in K$ with $f(x)=y$. Suppose that $x\notin\partial K$. Then there is an open neighborhood $V$ of $x$ such that $V\subseteq K$. Since $f$ is open it follows that $f(V)$ is an open neighborhood of $y=f(x)$ which is contained in $f(K)$ contradicting $y\in\partial f(K)$.
\end{proof}

\begin{proposition}\label{boundary distance}
Let $X\subseteq\R^n$ be open and $F\subseteq\R^{n+1}$ be closed. If $d_X$ satisfies the minimum principle in $F'$ then $d_{X\times\R}$ satisfies the minimum principle in $F$. Moreover, if $F$ is an affine subspace such that $\{x_{n+1};\,x\in F\}$ is bounded and $d_{X\times\R}$ satisfies the minimum principle in $F$ then $d_X$ satisfies the minimum principle in $F'$.
\end{proposition}

\begin{proof}
Let $d_X$ satisfies the minimum principle in $F'$. Let $K\subseteq F\cap (X\times\R)$ be compact. We assume that there is $x_0\in K$ such that
\[d_{X\times\R}(x_0)<\min_{x\in\partial_F K}d_{X\times\R}(x).\]
Since $d_{X\times\R}(x)=d_X(x')$ it follows
\begin{eqnarray*}
	d_X(x_0')&=&d_{X\times\R}(x_0)<\min_{x\in\partial_F K}d_{X\times\R}(x)\\
	&=&\min\{d_X(x');x'\in\pi_{|F}(\partial_F K)\}.
\end{eqnarray*}
Applying Proposition \ref{topology} to $\pi_{|F}F\rightarrow F'$ we obtain $\pi_{|F}(\partial_F K)\supseteq\partial_{F'}K'$ so that the previous inequality yields
\[d_X(x_0')<\min_{\partial_{F'}K'}d_X(x').\]
Since $x_0'\in K'$ and since $d_X$ satisfies the minimum principle this gives a contradiction. Thus
\[d_{X\times\R}(x_0)\geq\min_{x\in\partial_F K}d_{X\times\R}(x)\]
for every $x_0\in K$ so that $d_{X\times\R}$ satisfies the minimum principle in $F$.

Now assume that $F$ is an affine subspace such that $\{x_{n+1};\,x\in F\}$ is bounded and that $d_{X\times\R}$ satisfies the minimum principle in $F$. Since $\{x_{n+1};\,x\in F\}$ is bounded the subspace $F-F$ of $\R^{n+1}$ is contained in $\R^n\times\{0\}$. It follows for $x,y\in F$ with $\pi(x)=\pi(y)$ that $x'=y'$ and because of $x-y\in F-F\subseteq\R^n\times\{0\}$ we obtain $x=y$. Therefore, $\pi_{|F}$ is injective so that
\[\pi_{|F}:F\rightarrow F'\]
is a homeomorphism.

Let $K\subseteq F'\cap X$ be compact. It follows that $K\times\R\subseteq (F'\times\R)\cap (X\times\R)$, thus $(K\times\R)\cap F\subseteq F\cap(X\times\R)$. Because $\{x_{n+1};\,x\in F\}$ is bounded
\[L:=(K\times\R)\cap F\]
is a compact subset of $F\cap( X\times\R)$ with $\pi(L)=K$. Therefore, and since $d_{X\times\R}$ satisfies the minimum principle in $F$ we have
\begin{eqnarray*}
\min_{x'\in K}d_X(x')&=&\min\{d_X(x');\,x'\in\pi(L)\}=\min\{d_X(\pi(x));\,x\in L\}\\
&=&\min_{x\in L}d_{X\times\R}(x)=\min_{x\in\partial_F L}d_{X\times\R}(x)\\
&=&\min\{d_X(x');\,x'\in\pi_{|F}(\partial_F L)\}\\
&=&\min\{d_X(x');\,x'\in\partial_F' K\}=\min_{x'\in\partial_F' K}d_X(x'),
\end{eqnarray*}
where we used that $\pi_{|F}$ is a homeomorphism onto $F'$ and $\pi(L)=K$ in the last step.
\end{proof}

The next proposition will be useful in the proof of the following theorem.

\begin{proposition}\label{invariant subspace}
Let $P$ be a polynomial and let $V\subseteq\R^n$ be a subspace such that $P$ is constant on $V$. Then
\[\forall\,x\in\R^n,\xi\in V:\;P(x+\xi)=P(x),\]
i.e.\ $P$ acts along $V^\perp$.
\end{proposition}

\begin{proof}
By an appropriate linear change of coordinates we may assume without loss of generality that $V=\R^k\times\{0\}^{n-k}$ with $k=\mathrm{dim}V$. Since $P$ is constant on $V$ it follows
\[\forall\,\xi\in V, \alpha\in \N_0^n\cap V,\alpha\neq 0:\partial^\alpha P(\xi)=0.\]
Writing $\alpha\in\N_0^n$ as $\alpha=\alpha'+\alpha''$ with $\alpha'\in\N_0^n\cap V$ and $\alpha''\in\N_0^n\cap V^\perp$ it follows
\[\forall\,\xi\in V,\alpha\in\N_0^n, \alpha'\neq 0:\partial^\alpha P(\xi)=\partial^{\alpha''}\partial^{\alpha'} P(\xi)=0.\]
Since $\xi_{k+1}=\ldots =\xi_n=0$ for every $\xi\in V$ this implies together with Taylor's Theorem
\begin{eqnarray*}
	P(x+\xi)&=&\sum_{\alpha\in\N_0^n} \frac{\partial^\alpha P(0)}{\alpha!}(x+\xi)^\alpha=\sum_{\alpha\in\N_0^n\cap V^\perp} \frac{\partial^\alpha P(0)}{\alpha!}(x+\xi)^\alpha\\
	&=&\sum_{\alpha\in\N_0^n\cap V^\perp}\frac{\partial^\alpha P(0)}{\alpha!}x^\alpha=P(x)
\end{eqnarray*}
for every $x\in\R^n$ and $\xi\in V$.
\end{proof}

Recall that a polynomial $P$ is called semi-elliptic if it is possible to write
\[P(\xi)=\sum_{|\alpha:\mathbf{m}|\leq 1}c_\alpha \xi^\alpha\]
such that
\[\forall\,\xi\in\R^n\backslash\{0\}:\, \sum_{|\alpha:\mathbf{m}|= 1}c_\alpha \xi^\alpha\neq 0,\]
where $\mathbf{m}\in\N_0^n$ and $|\alpha:\mathbf{m}|=\sum_{j=1}^n\alpha_j/m_j$. With $m_1=1$ and $m_j=2$ for $j>1$ it follows that the polynomial inducing the heat operator is semi-elliptic.

\begin{theorem}\label{semi-elliptic}
Let $P\in\C[X_1,\ldots,X_n]$ be a semi-elliptic polynomial with principal part $P_m$ and let $X\subseteq\R^n$ be open. Then the following are equivalent.
\begin{itemize}
	\item[i)] $X\times\R$ is $P^+$-convex for singular supports.
	\item[ii)] $X$ is $P_m$-convex for supports.
	\item[iii)] $X$ is $P_m$-convex for singular supports.
	\item[iv)] $d_X$ satisfies the minimum principle in $x+\{y\in\R^n;\,P_m(y)=0\}^\perp$ for every $x\in\R^n$.
\end{itemize}
\end{theorem}

\begin{proof}
It is shown in \cite[Proposition 2, Lemma 3]{Frerick} that $Z:=\{x\in\R^n;\,P_m(x)=0\}$ is a subspace of $\R^n$ and that $\sigma_{P^+}(V)=0$ if and only if $V$ is a subspace of $Z\times\R$. In particular, 
\begin{equation}\label{inclusion}
\{x\in\R^{n+1};\sigma_{P^+}(x)=0\}= Z\times\R=(Z^\perp\times\{0\})^\perp.
\end{equation}

Since $\sigma_{P^+}(Z\times\R)=0$ by Lemma \ref{equivalence} it follows from \cite[Corollary 11.3.2]{Hoermander} that $d_{X\times\R}$ satisfies the minimum principle in $x+(Z\times\R)^\perp$ for every $x\in\R^{n+1}$ if $X\times\R$ is $P^+$-convex for singular supports. Proposition \ref{boundary distance} therefore implies that $d_X$ satisfies the minimum principle in every affine subspace parallel to $Z^\perp$ if $X\times\R$ is $P^+$-convex for singular supports. Thus, i) implies iv).

On the other hand, if $d_X$ satisfies the minimum principle in every affine subspace parallel to $Z^\perp$ it follows from Proposition \ref{boundary distance} that $d_{X\times\R}$ satisfies the minimum principle in every affine subspace parallel to $(Z\times\R)^\perp=Z^\perp\times\{0\}$. From (\ref{inclusion}) together with Theorem \ref{minimal sufficient} we therefore conclude that $X\times\R$ is $P^+$-convex for singular supports. Hence, iv) implies i).

In order to prove the remaining equivalences, we observe that by Proposition \ref{invariant subspace} $P_m$ is a polynomial acting along the subspace $Z^\perp$. Since $Z\cap Z^\perp=\{0\}$ $P_m$ is elliptic as a polynomial on $Z^\perp$ so that Theorem \ref{elliptic} yields the equvalences of ii), iii), and iv).
\end{proof}

We are now able to prove the main result of this section. Part iii) is in particular applicable to non-degenerate parabolic operators like the heat operator.

\begin{theorem}\label{special augmented}
Let $P$ be a polynomial with principal part $P_m$ and let $X\subseteq\R^n$ be open such that $P(D):\mathscr{D}'(X)\rightarrow\mathscr{D}'(X)$ is surjective. Then the augmented operator $P^+(D):\mathscr{D}'(X\times\R)\rightarrow\mathscr{D}'(X\times\R)$ is surjective if
\begin{itemize}
	\item[i)] $P$ acts along a subspace $W$ and is elliptic as a polynomial on $W$, or
	\item[ii)] $P(x)=\alpha\prod_{j=1}^l(\langle N_j,x\rangle+c_j)$ where $N_j\in\C^n\backslash\{0\}$ and $\alpha,c_j\in\C, j=1\ldots,l$, or
	\item[iii)] $P$ is a semi-elliptic polynomial for which $\{x\in\R^n;\,P_m(x)=0\}$ is a one-dimensional subspace.
\end{itemize}
\end{theorem}

\begin{proof}
Since $X\times\R$ is $P^+$-convex for supports whenever $X$ is $P$-convex for supports by \cite[Proposition 1]{Frerick} we only have to show that $X\times\R$ is $P^+$-convex for singular supports.

If $P$ acts along a subspace and is elliptic there, the same holds for $P^+$. As $X\times\R$ is $P^+$-convex for supports it follows from Theorem \ref{elliptic} that $X\times\R$ is also $P^+$-convex for singular supports, proving i).

For $P(x)=\alpha\prod_{j=1}^l(\langle N_j,x\rangle+c_j)$, each of the factors of $P$ acts along a subspace and is elliptic there. Thus applying i) to each of the factors yields ii). 

In case of iii) it follows from \cite[Theorem 10.8.1]{Hoermander} that $d_X$ satisfies the minimum principle in every characteristic hyperplane. Thus, it follows from Theorem \ref{semi-elliptic} that $X\times\R$ is $P^+$-convex for singular supports.
\end{proof}

\begin{remark}
\begin{rm}
a) It was kindly pointed out to us by D.\ Vogt that for $P(D)=\langle D,N\rangle +c$ in case of $N\in\R^n$ the surjectivity of $P(D)$ on $\mathscr{D}'(X)$ already implies the existence of a continuous linear right inverse of $P(D)$ on $\mathscr{D}'(X)$ (see \cite{MeiseTaylorVogt}) and thus a continuous linear right inverse of $P^+(D)$ on $\mathscr{D}'(X\times\R)$. In particular, $P^+(D)$ is then surjective on $\mathscr{D}'(X\times\R)$. However, due to the well-known result of Grothendieck that elliptic $P(D)$ do not have continuous linear right inverses for $n\geq 2$ (see e.g.\ \cite[Theorem C.1]{Treves}) on $C^\infty(X)$ combined with \cite[Theorem 2.7]{MeiseTaylorVogt} the Cauchy-Riemann operator $\partial_{\overline{z}}$ shows that on $\mathscr{D}'(X)$ a surjective first order operator $P(D)=\langle D,N\rangle +c$ with $N\in\C^n\backslash\R^n$ in general does not have a continuous linear right inverse.

\noindent b) As shown in the proof of Theorem \ref{semi-elliptic}, the principal part $P_m$ of a semi-elliptic polynomial $P$ acts along the subspace $\{x\in\R^n;\,P_m(x)=0\}^\perp$ and is elliptic as a polynomial on this subspace. Therefore, applying Theorem \ref{special augmented}, Theorem \ref{elliptic}, and Proposition \ref{boundary distance}, it follows that for a semi-elliptic polynomial the statement
\begin{itemize}
	\item[v)] $X\times\R$ is $P_m^+$-convex for singular supports.
\end{itemize}
can be added to the equivalences in Theorem \ref{semi-elliptic}.
\end{rm}
\end{remark}

\section{Surjectivity of differential operators on vector-valued functions and distributions and linear topological invariants}\label{topological invariants}

In this section we consider the problem whether surjectivity of a differential operator $P(D)$ on $C^\infty(X)$, resp.\ on $\mathscr{D}'(X)$, passes on to surjectivity of $P(D)$ on spaces of vector-valued functions, resp.\ vector-valued distributions. As usual, for a locally convex space $E$ we denote by $C^\infty(X,E)$ the space of smooth functions with values in $E$ and by $\mathscr{D}'(X,E)$ the space of $E$-valued distributions, i.e.\ the space of continuous linear operators from $\mathscr{D}(X)$ to $E$ equipped with its usual locally convex topology. As explained in the introduction, surjectivity of $P(D)$ on spaces of vector-valued smooth functions and distributions is connected with the problem of parameter dependence for solutions of the equation $P(D)u_\lambda=f_\lambda$ on the corresponding scalar-valued spaces. Obviously, for two topologically isomorphic locally convex spaces $E_1$ and $E_2$, $P(D)$ is surjective on $C^\infty(X,E_1)$ (resp.\ $\mathscr{D}'(X,E_1)$) if and only if $P(D)$ is surjective on $C^\infty(X,E_2)$ (resp.\ $\mathscr{D}'(X,E_2)$). As described in the introduction, this problem is addressed by proving that the kernels $C_P^\infty(X)$, resp.\ $\mathscr{D}_P'(X)$ have properties $(\Omega)$, resp.\ $(P\Omega)$ when $P(D)$ is surjective on $C^\infty(X)$, resp.\ $\mathscr{D}'(X)$.

\begin{theorem}\label{properties omega}
	Let $X\subseteq\R^n$ be open and let $P$ be a polynomial with principle part $P_m$.
	\begin{itemize}
		\item[a)] Let $P$ be semi-elliptic such that $\{x\in\R^n;\,P_m(x)=0\}$ is a one-dimensional subspace of $\R^n$. If $X$ is $P$-convex for supports then $C_P^\infty(X)$ has property $(\Omega)$.
		\item[b)] Let $P$ act along a subspace of $\R^n$ such that it is elliptic there. If $P(D)$ is surjective on $\mathscr{D}'(X)$ then $\mathscr{D}'_P(X)$ has property $(P\Omega)$.
		\item[c)] Let $P(x)=\alpha\prod_{j=1}^l(\langle N_j,x\rangle -c_j)$ where $\alpha,c_j\in\C$ and $N_j\in \C^n\backslash\{0\}, j=1,\ldots,l$ and let $P(D)$ be surjective on $\mathscr{D}'(X)$. Then $\mathscr{D}'_P(X)$ has $(P\Omega)$.
	\end{itemize}
\end{theorem}

\begin{proof}
	First, let $P$ be a semi-elliptic polynomial such that $\{x\in\R^n;\,P_m(x)=0\}$ is a one-dimensional subspace of $\R^n$. Since semi-elliptic operators are hypoelliptic it follows that $P(D)$ is surjective on $\mathscr{D}'(X)$ if it is surjective on $C^\infty(X)$. Thus, by Theorem \ref{special augmented} it follows that $P^+(D)$ is surjective on $\mathscr{D}'(X\times\R)$ so that $\mathscr{D}'_P(X)$ has $(P\Omega)$ by \cite[Proposition 8.3]{Bonet}. Since $P(D)$ is hypoelliptic, $\mathscr{D}'_P(X)$ and $C^\infty_P(X)$ coincide as locally convex spaces by \cite[Theorem 52.1]{Treves2}, thus the Fr\'echet-Schwartz space $C^\infty_P(X)$ has $(\Omega)$ by \cite[Proposition 5.3]{Bonet}. This proves a).
	
	For a polynomial $P$ which acts along a subspace and is elliptic there, using muatis mutandis the same arguments used to prove a) yield b).
	
	Finally, let $P(x)=\alpha\prod_{j=1}^l(\langle N_j,x \rangle -c_j)$ and let $P(D)$ be surjective on $\mathscr{D}'(X)$. Each factor of $P$ acts along one of the subspaces $\mbox{span}\{\mbox{Re}\,N_j,\mbox{Im}\,N_j\}$ and is elliptic there so it follows from part b) and \cite[Proposition 8.3]{Bonet} that each of the factors of $P(D)$ and therefore $P(D)$ itself is surjective on $\mathscr{D}'(X\times\R)$. Another reference to \cite[Proposition 8.3]{Bonet} yields that $\mathscr{D}'_P(X)$ has $(P\Omega)$.
\end{proof}

\begin{corollary}
	Let $X\subseteq\R^n$ be open and let $P$ be a polynomial with principle part $P_m$.
	\begin{itemize}
		\item[a)] Let $F$ be a Fr\'echet space with property $(DN)$ and let $P$ be semi-elliptic such that $\{x\in\R^n;\,P_m(x)=0\}$ is a one-dimensional subspace of $\R^n$ and such that $P(D)$ is surjective on $C^\infty(X)$. Then $P(D)$ is surjective on $C^\infty(X,F')$.
		\item[b)]  Let $E$ be a locally convex space such that $E$ is either the strong dual of a nuclear Fr\'echet space with $(DN)$, the strong dual of a Fr\'echet-Schwartz space with $(DN)$ which has an absolute Schauder basis, or $\mathscr{D}'(Y)$, where $Y\subseteq\R^m$ is open.    
		\begin{itemize}
			\item[i)] Let $P$ be semi-elliptic such that $\{x\in\R^n;\,P_m(x)=0\}$ is a one-dimensional subspace of $\R^n$ and such that $P(D)$ is surjective on $C^\infty(X)$. Then $P(D)$ is surjective on $\mathscr{D}'(X,E)$.
			\item[ii)] Assume $P$ acts along a subspace of $\R^n$ and is elliptic there such that $P(D)$ is surjective on $\mathscr{D}'(X)$. Then $P(D)$ is surjective on $\mathscr{D}'(X,E)$.
			\item[iii)] Let $P(x)=\alpha\prod_{j=1}^l(\langle N_j,x\rangle -c_j)$ where $\alpha,c_j\in\C$ and $N_j\in \C^n\backslash\{0\}, j=1,\ldots,l$ such that $P(D)$ is surjective on $\mathscr{D}'(X)$. Then $P(D)$ is surjective on $\mathscr{D}'(X,E)$.
		\end{itemize}
		\end{itemize}
\end{corollary}

\begin{proof}
	If the hypothesis of a) are satisfied it follows from part a) of Theorem \ref{properties omega} that $C_P^\infty(X)$ has property $(\Omega)$. If $F$ is a Fr\'echet space with $(DN)$ it follows from \cite[Theorem 2.4 b)]{Vogt} that $P(D)$ is surjective on $C^\infty(X,F')$ which proves a). In case $F$ is a Fr\'echet-Schwartz space with absolute Schauder basis, $F$ is topologically isomorphic to a K\"othe sequence space $\lambda^1(A)$ (see \cite[Proposition 27.26]{MeiseVogt}). If $F$ additionally has $(DN)$ it follows from \cite[Theorem 5.6, Corollary 4.2, and Proposition 3.3]{Bonet} that $P(D)$ is surjective on $\mathscr{D}'(X,F')$. In case $F$ is a nuclear Fr\'echet space with $(DN)$ it follows again from \cite[Theorem 5.6, Corollary 4.2, and Proposition 3.3]{Bonet} that $P(D)$ is surjective on $\mathscr{D}'(X,F')$. Since by the Schwartz-Kernel Theorem $\mathscr{D}'(X\times\R)$ and $\mathscr{D}'(X,\mathscr{D}'(\R))$ are canonically isomorphic and since under this isomorphism $P^+(D)$ and $P(D)$ on $\mathscr{D}'(X,\mathscr{D}'(\R))$ are conjugate it follows that the latter is surjective. Moreover, since $\mathscr{D}'(\R)$ and $\mathscr{D}'(Y)$ are topologically isomorphic (see \cite{Vogt2}) we finally obtain surjectivity of $P(D)$ on $\mathscr{D}'(X,\mathscr{D}'(Y))$ which proves part i) of b).
	
	For a polynomial $P$ which acts along a subspace and is elliptic there, the same arguments used to prove b) i) yield part ii) of b).
	
	Finally, let $P(x)=\alpha\prod_{j=1}^l(\langle N_j,x \rangle -c_j)$ such that $P(D)$ is surjective on $\mathscr{D}'(X)$. By part b) of Theorem \ref{properties omega} $\mathscr{D}_P'(X)$ has $(P\Omega)$. Thus, applying \cite[Proposition 8.3]{Bonet} and the Schwartz-Kernel Theorem once more gives b) iii).
\end{proof}

\begin{small}
{\sc Technische Universit\"at Chemnitz,
Fakult\"at f\"ur Mathematik, 09107 Chemnitz, Germany}

{\it E-mail address: thomas.kalmes@mathematik.tu-chemnitz.de}
\end{small}

\end{document}